\documentclass[11pt,reqno]{amsart}
\usepackage{amsfonts}
\usepackage{ifthen}
\usepackage{amsthm}
\usepackage{amsmath}
\usepackage{graphicx}
\usepackage{caption}
\usepackage{subcaption}
\usepackage{amscd,amssymb}
\usepackage{color}
\usepackage{hyperref}
\usepackage{array,multirow,makecell}

\usepackage{cite,epstopdf}

\setcellgapes{1pt}
\makegapedcells
\newcolumntype{R}[1]{>{\raggedleft\arraybackslash }b{#1}}
\newcolumntype{L}[1]{>{\raggedright\arraybackslash }b{#1}}
\newcolumntype{C}[1]{>{\centering\arraybackslash }b{#1}}
\newcounter{minutes}
\divide\time by 60
\newcounter{hours}
\multiply\time by 60 \addtocounter{minutes}{-\time}

\newtheorem{theorem}{Theorem}
\newtheorem{lemma}{Lemma}

\newtheorem{corollary}{Corollary}
\newtheorem{remark}{Remark}
\newtheorem{example}{Example}

\numberwithin{equation}{section}

\title[Subordination Associated with Laguerre polynomial]{Subordination Associated with Laguerre polynomial}

\author[A. Kumar]
{Anish Kumar}

\address{{\bf Anish Kumar}\newline
Department of Mathematics,
Dr. Shyama Prasad Mukherjee University ,\newline
Ranchi 834008, Jharkhand, India}
\email{ak8107690@gmail.com}

\begin{document}

\begin{abstract}
In this work, we have considered the Laguerre polynomial. This polynomial has been studied in several branches of theoretical physics and applied Mathematics. J. K. Prajapat at.al derived condition so that Laguerre polynomial satisfy convexity, strong starlikeness, close-to-convexity and strongly convexity. In this article, characteristics properties such as exponential subordination have been studied.  Moreover Janowski starlikeness and convexity have been investigated for this polynomial. Several examples and corollaries have been mentioned to validates the result. 
\end{abstract}
\maketitle

\section{Introduction and Motivation}
Let $\mathcal A$ denote the family of normalized analytic functions in the unit disk $\mathbb D=\{z\in\mathbb C:|z|<1\}$ satisfying
$
F(0)=0 \quad \text{and} \quad F'(0)=1.
$
A function $F\in\mathcal A$ is said to be univalent in $\mathbb D$ if it is one-to-one there. In geometric function theory, the quantities
$
\frac{zF'(z)}{F(z)}
$
and
$
1+\frac{zF''(z)}{F'(z)}
$
are fundamental in characterizing starlikeness and convexity, respectively. In particular, when either of these quantities lies in the domain
$
|\log U|<1,
$
associated with the exponential mapping, one obtains subclasses related to exponential starlikeness and exponential convexity.

Subordination provides a natural framework for studying these geometric properties. For analytic functions $g$ and $h$ in $\mathbb D$, one says that $g$ is subordinate to $h$, written $g\prec h$, if there exists a Schwarz function $w$, analytic in $\mathbb D$ with $w(0)=0$ and $|w(z)|<1$, such that
$
g(z)=h(w(z)), \quad z\in\mathbb D.
$
Equivalently, when $h$ is univalent,
$
g(0)=h(0)
$
and
$
g(\mathbb D)\subset h(\mathbb D).
$

Let $\mathcal P_e$ denote the class of analytic functions $p$ in $\mathbb D$ satisfying
$
p(0)=1
$
and
$
p(z)\prec e^z.
$
Then a function $F\in\mathcal A$ is called exponentially starlike if
$
\frac{zF'(z)}{F(z)}\in \mathcal P_e,
$
and exponentially convex if
$
1+\frac{zF''(z)}{F'(z)}\in \mathcal P_e.
$
These classes are denoted by $\mathcal S_e^*$ and $\mathcal K_e$, respectively.

Motivated by a unified treatment of geometric subclasses, Ma and Minda introduced general subclasses of starlike and convex functions, including the classical families

$$\mathcal{S}^*=\left\{F\in \mathcal{A}: 0<  \Re\left(\frac{zF^\prime(z)}{F(z)}\right) \quad
 \forall \; z\in\mathbb{D}\right\},$$
$$\mathcal{K}=\left\{F\in \mathcal{A}:  0<\Re\left(1+\frac{zF^{\prime\prime}(z)}{F^{\prime}(z)}\right)\quad
\mbox{for all} \; z\in\mathbb{D}\right\}.$$ 

Another important subclass arises from the Bernoulli lemniscate. A function is called lemniscate starlike whenever
$
\frac{zF'(z)}{F(z)} \prec \sqrt{1+z},
$
while lemniscate convexity is characterized by
$
1+\frac{zF''(z)}{F'(z)} \prec \sqrt{1+z}.
$

For parameters satisfying
$
-1\le D<C\le 1,
$
let $P(C,D)$ denote the family of analytic functions
$
p(z)=1+c_1z+c_2z^2+\cdots
$
such that
$
p(z)\prec \frac{1+Cz}{1+Dz}.
$
In particular, for $0\le \beta<1$, the choice
$
P(1-2\beta,-1)
$
corresponds to functions satisfying
$
\Re(p(z))>\beta.
$

A function $F\in\mathcal A$ belongs to the Janowski starlike class $S^*(C,D)$ if
$
\frac{zF'(z)}{F(z)}\in P(C,D).
$
These subclasses provide a broad framework for investigating geometric properties through differential subordination and form the foundation for the results developed in this article.

The extended  Sonine polynomials or Laguerre polynomials  of degree n and  parameter $\alpha \in \mathbb{R}$ are given using Rodrigues formulae \\
$$L_n^{(\alpha)}(z)= \frac{z^{-\alpha}e^z}{n!}\left(\frac{d}{dz}\right)^n(z^{n+\alpha}e^{-z}).$$

They hold the condintion of secnd order differential equation,
\begin{equation}\label{eq1}
zL^{\prime\prime}+(\alpha+1-z)L'(z)+nL(z)=0.
\end{equation}
Particularly using $\alpha=0$ is the Laguerre polynomial $L_n(z)$ which has been discussed in article. For $-1<\alpha$ and $n\in 0\cup \mathbb{N}$, the extended Laguerre polynomial $L_n(\alpha)(z)$ can be expressed as \\
\begin{align}
   L_n(\alpha)(z) &=\sum_{k=0}^{n}(-1)^k\binom{n+\alpha}{n-k}\frac{z^k}{k!}\\
   &=\sum_{k=0}^{n}\frac{(-1)^k (1+\alpha)_n}{(1+\alpha)_k(n-k)!}\frac{z^k}{k!};
\end{align}
where $(z)_n=\frac{\Gamma(z+n)}{\Gamma(z)}$, $n\in {0} \cup \mathbb{N}$ is a pochhamer symbol. \\

The Laguerre polynomials are vital in applied and theoretical physics as well as Mathematics. They arise in quantum mechanics, particularly in the radial component of the solution to the Schrödinger equation for a one-electron atom.
 They represent the static winger function of oscillator systems and quantum mechanics too. The Laguerre polynomials not reducible, one may refer to (cited there in). 

The analytical investigation of the Laguerre polynomials have been broadly done. Never the less its behaviour as an analytic functions has not explored much. The geometric behaviour of special functions were explored several researcher such discussions are apperared for the bessel function \cite{Baricz, Arpad}, hypergeometric function \cite{hasto,ssm}, Wright function \cite{baricz}, Mittag-Leffler function \cite{Raina, bansal}, Coulumb wave function \cite{mondal}, incomplete beta function \cite{anbhu}, stuve and Lommel function \cite{baricz}.

For $-1<\alpha $ and $n \in   {0} \cup \mathbb{N}$, the normalize Laguerre polynomials are defined as\\
$$M_{n, \alpha}(z)=\frac{n!}{(1+\alpha)_n}L_n^{\alpha}(z)=\sum_{k=0}^n \frac{(-1)^k n!}{(1+\alpha)_k (n-k)!}\frac{z^k}{k!},$$ Which hold differential equation too. If we put $n=0,1,2,3, \dots$ we can get the some terms of $M_{n,\alpha}$;
\begin{align*}
 & M_{0, \alpha} (z)=1 , M_{1,\alpha}(z)=1-\frac{z}{1+\alpha}\\
 &M_{2,\alpha}(z)=1-\frac{2z}{1+\alpha}+\frac{z^2}{(1+\alpha)(2+\alpha)}\\
 & M_{3,\alpha}=1-\frac{3z}{1+\alpha} +\frac{3z^2}{(1+\alpha)(2+\alpha)}-\frac{z^3}{(1+\alpha)(2+\alpha)(3+\alpha)}.
\end{align*}
This polynomial also related to special functions like special functions like special function $J_{\alpha}(z)$, Whitteker function $W_{a,b}(z)$, hypergeometric function and Hermite polynomial.

\begin{align*}
    &M_{n, \alpha}(z)=\phi(-n;\alpha+1;z)\\
    &M_{n,\alpha}(z)=z^{\frac{-1}{2}(\alpha+1)}e^\frac{z}{2}W_{n+\frac{1}{2},\frac{\alpha}{2}}(z),\\
    &\sum_{0}^{\infty}M_{n,\alpha}(z)\frac{t^n}{n!}=\Gamma(1+\alpha)(zt)^{\frac{-\alpha}{2}e^t J_{\alpha}(2\sqrt zt)}\\
    &M_{n,-1/2}(z^2)=\frac{(-1)^n n!}{2n!}H_{2n}(z).
\end{align*}
In addition, we can find that $$M_{n,\alpha}^{\prime}(z)=-\frac{n}{(1+\alpha)}M_{n-1, \alpha+1}(z).$$


In order to show the main findings, we need lemmas, which are as follows:

\section{Useful Lemmas}\label{sec1}
  Some Lemmas have been recollected in the below section, which will be useful to show our main results. 
\begin{lemma}\label{lem1}\rm{\cite{Naz}}
Let $\Omega\subset \mathbb C$ and let
$
\psi:\mathbb C^3\times \mathbb D \to \mathbb C
$
be such that
$
\psi(r,s,t;z)\notin \Omega
$
whenever
$
r=e^{e^{i\theta}},
$
$
s=me^{i\theta}e^{e^{i\theta}},
$
and
$
\Re\left((s+t)e^{-i\theta}e^{-e^{i\theta}}\right)\ge 0,
$
for $\theta\in[0,2\pi)$, $z\in\mathbb D$, and $m\ge 1$. If $q$ is analytic in $\mathbb D$ with $q(0)=1$ and
$
\psi(q(z),zq'(z),z^2q''(z);z)\in \Omega,
$
for all $z\in\mathbb D$, then
$
q\in \mathcal P_e.
$
\end{lemma}

It is worth observing that the admissibility requirement in Lemma \ref{lem1} is verified whenever
$
\psi(r,s,t;z)\notin \Omega
$
for
$
r=e^{e^{i\theta}},
$
$
s=me^{i\theta}e^{e^{i\theta}},
$
and
$
\Re\left((s+t)e^{-i\theta}e^{-e^{i\theta}}\right)\ge 0,
$
where $\theta\in[0,2\pi)$, $z\in\mathbb D$, and $m\ge1$.

In the special case
$
\psi:\mathbb C^2\times\mathbb D\to\mathbb C,
$
the admissibility condition reduces to
$
\psi(e^{e^{i\theta}},me^{i\theta}e^{e^{i\theta}};z)\notin\Omega,
$
where $z\in\mathbb D$, $\theta\in[0,2\pi)$, and $m\ge1$.

\begin{lemma}\label{lem2}\rm{\cite{noreen}}
Suppose $\Omega\subset\mathbb C$ and let
$
\psi:\mathbb C^3\times\mathbb D\to\mathbb C
$
satisfy
$
\psi(i\rho,\sigma,\mu+iv;z)\notin\Omega
$
whenever $z\in\mathbb D$, $\rho\in\mathbb R$,
$
\sigma\le -\frac{1+\rho^2}{2},
$
and
$
\sigma+\mu\le0.
$
If $q$ is analytic in $\mathbb D$ with $q(0)=1$ and
$
\psi(q(z),zq'(z),z^2q''(z);z)\in\Omega,
$
then
$
\Re(q(z))>0
$
for all $z\in\mathbb D$.
\end{lemma}

For the reduced case
$
\psi:\mathbb C^2\times\mathbb D\to\mathbb C,
$
the admissibility condition in Lemma \ref{lem2} becomes
$
\psi(i\rho,\sigma;z)\notin\Omega,
$
where $\rho\in\mathbb R$ and
$
\sigma\le -\frac{1+\rho^2}{2}.
$

\begin{lemma}\label{lem3}\rm{\cite{Madaan}}
Assume $q$ belongs to the normalized class of analytic functions and satisfies
$
q(z)\neq 1.
$
Let $\Omega\subset\mathbb C$ and let
$
\psi:\mathbb C^3\times\mathbb D\to\mathbb C
$
be such that
$
\psi(r,s,t;z)\notin\Omega
$
whenever
$
r=\sqrt{2\cos2\theta},e^{i\theta},
$
$
s=\frac{me^{3i\theta}}{2\sqrt{2\cos2\theta}},
$
and
$
\Re\left((s+t)e^{-3i\theta}\right)
\ge
\frac{3m^2}{8\sqrt{2\cos2\theta}},
$
for
$
-\frac{\pi}{4}<\theta<\frac{\pi}{4},
$
$m\ge n\ge1$, and $z\in\mathbb D$.

If
$
(q(z),zq'(z),z^2q''(z);z)\in\mathbb D
$
and
$
\psi(q(z),zq'(z),z^2q''(z);z)\in\Omega,
$
then
$
q(z)\prec \sqrt{1+z}.
$
\end{lemma}

In this case $\psi:\mathbb{C}^{2}\times \mathbb{D}\rightarrow \mathbb{C}$, the condition in Lemma \ref{lem3}, reduces to $\psi({r,s;z})$ whenever $r=\sqrt{2\cos2\theta}e^{i\theta}, s=\frac{me^{3i\theta}}{2\sqrt{2\cos2\theta}}$, $m\geq n\geq 1$ and $-\frac{\pi}{4}<\theta<\frac{\pi}{4}$ and $z\in \mathbb{D}$.

\section{Exponential Subordination }
In this section main results related to exponential subordination involving Laguerre polynomial has been investigated.
\begin{theorem}\label{thm1}
  Assume that $\alpha >-1$, $n \in \mathbb{N} \cup{0}$ and hold the condition $\frac{1}{e}[\Re (\alpha-1)-n]>0$, then $M_{n,\alpha}(z) \in \mathcal{P}_e$.
\end{theorem}
\begin{proof}
   Let $q: \mathbb{D} \rightarrow \mathbb{C}$ be given by $q(z)=M_{n,\alpha}(z).$
   Here $q(z)$ is holomorphic function with $q(0)=1.$ By equation \ref{eq1}, $q(z)$ hold the second order differential equation, we get \\
   $$z^2q^{\prime\prime}(z) +(\alpha+1-z)zq^{\prime}(z)+znq(z)=0.$$
   Let us consider an other function $\psi : \mathbb{C}^3\times \mathbb{D} \rightarrow \mathbb{C}$ by  $$\psi(r,s,t:z)=t+(\alpha+1-z)s+rnz.$$ Assume that $\Delta =0$. Then $\psi(q(z),zq'(z),z^2 q^{\prime\prime}:z) \in \Delta$, $\forall z\in \mathbb{D}.$\\
   Now, we prove that $q(z) \prec e^z$, using lemma \ref{lem1}. We have $s=me^{i\theta}e^{e^{i\theta}}, r=e^{e^{i\theta}} $, $\Re((s+t)e^{-i\theta}e^{-e^{i\theta}})\geq 0$, $z\in \mathbb{D}, \theta \in [0,2\pi), $ and $1\leq m$. Remind triangle inequality $||z_1|-|z_2||\leq |z_1-z_2|$, for all $z_2,z_1 \in \mathbb{C}$ and take into account
   \begin{align*}
      | \psi(r,s,t;z)|&=|e^{e^{i\theta}}||(t+s)e^{-e^{i\theta}}+(\alpha-z)me^{e^{i\theta}}+nz|\\
      &\geq e^{\cos\theta}||(t+s)e^{-i\theta}e^{-e^{i\theta}}+(\alpha-z)m|-|nz||\\
      &> \frac{1}{e}[\Re (\alpha-1)-n].
   \end{align*}
   By given hypothesis $\frac{1}{e}[\Re (\alpha-1)-n]>0$.
 We get the desired result. 
   \end{proof}
   
  \begin{remark}
      After setting particular value of $\alpha$, $n$ and by theorem \ref{thm1}, we obtain following examples which are subordinate to exponential function $e^z$: \begin{example}
           $M_{1,3}(z)=1-\frac{z}{4}$
           \end{example}
\begin{example}
$M_{2,4}(z)=1-\frac{2z}{5}+\frac{z^2}{40}$
\end{example}
\begin{example}
  $ M_{3,5}=1-\frac{3z}{6} +\frac{3z^2}{42}-\frac{z^3}{336}$
      \end{example}
     \begin{example}
      Hypergeometric function $M_{2,4}=\phi(-2,5;z)$.   
     \end{example}
     \begin{example}
         Wright function $M_{3,5}=z^{-3}e^{\frac{z}{2}}W_{6,\frac{5}{2}}(z)$
     \end{example}
     \begin{example}
         Hermite polynomial $M_{4,\frac{-5}{2}}(z^2)=\frac{H_8(z)}{8.7.6.5}$.
     \end{example}
  \end{remark}
  
In the next section, we have obtained sufficient conditions such that  Laguerre polynomial belongs to Janowski convexity. Further we will show that $M_{n,\alpha}(z)$ belong to Janowski starlikeness $S^{*}[C,D]$.
\section{Janowski convexity and starlikeness of Laguerre polynomial}\label{sec3}

\begin{theorem}\label{thm2}
Let $\alpha>-1 ,n \in \mathbb{N}$,\begin{align*}
h_1
&=\Big[-(C-D)(1+D)-(n+1)(1+D)^2\Big]
\Big[-(C-D)(1-D)+(n+1)(1-D)^2\Big]>0,
\end{align*}

\begin{align*}
h_2
&=\Big[(C-D)+(C-D)^2+(C-D)(1+D)(\alpha+1)\Big]
\Big[-(C-D)(1-D)+(n+1)(1-D)^2\Big] \\
&\quad + \Big[-(C-D)(1+D)-(n+1)(1+D)^2\Big]
\Big[(C-D)-(C-D)^2+(C-D)(1-D)(\alpha+1)\Big],
\end{align*}

\begin{align*}
h_3
&=\Big[(C-D)+(C-D)^2+(C-D)(1+D)(\alpha+1)\Big]
\Big[(C-D)-(C-D)^2+(C-D)(1-D)(\alpha+1)\Big].
\end{align*} and $-1\leq D < C \leq 1$. 
Whenever \begin{align*}
    (C-D)(1+D)+(n+1)(1+D)^2 >0,
\end{align*}
inequality hold 
\begin{align}\label{eq4}
2(C-D)+2(C-D)^2 +(C-D)(1+D)\alpha -(n+1)(1+D)^2>0.\end{align}
Whenever \begin{align*}
    (C-D)(1+D)+(n+1)(1+D)^2 <0,
\end{align*}
inequality hold 
\begin{align}\label{eq5}
2(C-D)+2(C-D)^2 +(C-D)(1+D)(\alpha+2) +(n+1)(1+D)^2>0.\end{align}
Whenever $|\frac{-h_2}{2h_1}|<1$, inequality hold \begin{align*}
&\max \Big\{
\big[-(C-D)^2 - D(C-D)(\alpha+2) + (n+1)(1-D^2)\big]^2,\\
&\qquad
\big[(C-D)^2 + D(C-D)\alpha + (n+1)(1-D^2)\big]^2
\Big\}\\
&< h_3-\frac{h_2^{2}}{4h_1}.
\end{align*}
Whenever $2h_1 +h_2 \leq 1$, inequality satisfy 
\begin{align*}
&\max \Big\{
\big[-(C-D)^2 - D(C-D)(\alpha+2) + (n+1)(1-D^2)\big]^2,\\
&\qquad
\big[(C-D)^2 + D(C-D)\alpha + (n+1)(1-D^2)\big]^2
\Big\}\\
&< h_1 +h_2+h_3.
\end{align*}

If $0\notin M_{n,\alpha}'(\mathbb{D})$ and $0\notin M_{n,\alpha}''(\mathbb{D})$, then
$$1+z\frac{M_{n,\alpha}''(z)}{M_{n,\alpha}'(z)} \prec \frac{1+Cz}{1-Dz}.$$
\end{theorem}

\begin{proof}

Define $q: \mathbb{D} \rightarrow \mathbb{C}$
$$q(z)=\frac{(C-D)\phi'(z)+(1-D)z\phi''(z)}
{(C-D)\phi'(z)-(1+D)z\phi''(z)},$$ where $\phi(z)=M_{n,\alpha}(z).$

Then

$$q(z)\Big((C-D)\phi'(z)-(1+D)z\phi''(z)\Big)\ =(C-D)\phi'(z)+(1-D)z\phi''(z),$$

which gives
\begin{align*}
(C-D)\phi'(z)(q-1)
&=z\phi''(z)\Big((q+1)+D(q-1)\Big).
\end{align*}

Thus
$$\frac{z\phi''(z)}{\phi'(z)}
=\frac{(C-D)(q(z)-1)}
{(q(z)+1)+D(q(z)-1)}.$$

\medskip

Taking logarithmic derivative,
\begin{align}
\frac{d}{dz}\log\left(\frac{z\phi''}{\phi'}\right)
=\frac{d}{dz}\log\left(\frac{(C-D)(q-1)}
{(q+1)+D(q-1)}\right).
\end{align}

Left hand side becomes
$$\frac{1}{z}+\frac{\phi'''(z)}{\phi''(z)}
-\frac{\phi''(z)}{\phi'(z)}.$$

Right hand side becomes
$$\frac{q'(z)}{q(z)-1}
-\frac{(D+1)q'(z)}
{(q(z)+1)+D(q(z)-1)}.$$

Thus
\begin{align}
\frac{1}{z}+\frac{\phi'''}{\phi''}
-\frac{\phi''}{\phi'}
=\frac{q'(z)}{q(z)-1}
-\frac{(D+1)q'(z)}{(q(z)+1)+D(q(z)-1)}.
\end{align}

Multiplying by $z$ both sides,
\begin{align}
1+\frac{z\phi'''}{\phi''}
-\frac{z\phi''}{\phi'}
=\frac{zq'(z)}{q(z)-1}
-\frac{(D+1)zq'(z)}{(q(z)+1)+D(q(z)-1)}.
\end{align}

Hence
$$
\frac{z\phi'''}{\phi''}
=\frac{2zq'(z)}{(q(z)-1)((q(z)+1)+D(q(z)-1))}\
-1+\frac{z\phi''}{\phi'}.$$

Multiplying,
$$
\frac{z\phi'''(z)}{\phi''(z)}
\frac{z\phi''(z)}{\phi'(z)}
=\frac{2(C-D)zq'(z)}
{((q(z)+1)+D(q(z)-1))^2}\
-\frac{(C-D)(q(z)-1)}
{(q(z)+1)+D(q(z)-1)}\
+\frac{(C-D)^2(q(z)-1)^2}
{((q(z)+1)+D(q(z)-1))^2}.$$

\medskip

Now consider differential equation,
$$z\phi''(z)+(\alpha+1-z)\phi'(z)-n\phi(z)=0.$$

Differentiating,
\begin{align*}
\phi''(z)+z\phi'''(z)-\phi'(z)
+(\alpha+1-z)\phi''(z)-n\phi'(z)=0.
\end{align*}

Thus
$$z\phi'''(z)+(\alpha+2-z)\phi''(z)
-(n+1)\phi'(z)=0.$$

Dividing by $\phi'(z)$,
$$z\frac{\phi'''(z)}{\phi'(z)}
+(\alpha+2-z)\frac{\phi''(z)}{\phi'(z)}
-(n+1)=0.$$

Multiplying by $z$,
\begin{align}
z^2\frac{\phi'''(z)}{\phi'(z)}
+(\alpha+2-z)z\frac{\phi''(z)}{\phi'(z)}
-(n+1)z=0.
\end{align}

Using substitution,
$$\left(\frac{z\phi'''(z)}{\phi''(z)}
\frac{z\phi''(z)}{\phi'(z)}\right) 
+(\alpha+2-z)\frac{z\phi''(z)}{\phi'(z)}
-(n+1)z=0.
$$

\begin{align*}
&=\frac{2(C-D)zq'(z)}{\big((q(z)+1)+D(q(z)-1)\big)^2}
-\frac{(C-D)(q(z)-1)}{(q(z)+1)+D(q(z)-1)}  +\frac{(C-D)^2(q(z)-1)^2}{\big((q(z)+1)+D(q(z)-1)\big)^2} \\
& +(\alpha+2-z)\frac{(C-D)(q(z)-1)}{(q(z)+1)+D(q(z)-1)}  -(n+1)z \\
&=\frac{2(C-D)zq'(z)}{\big((q(z)+1)+D(q(z)-1)\big)^2}
+\frac{(C-D)^2(q(z)-1)^2}{\big((q(z)+1)+D(q(z)-1)\big)^2}  +\frac{(C-D)(q(z)-1)(\alpha+1-z)}{(q(z)+1)+D(q(z)-1)} \\
&\quad -(n+1)z =0.
\end{align*}

Let us simplify
\[
A=(q(z)+1)+D(q(z)-1)=(1+D)q+(1-D).
\]

Multiplying throughout by $A^2$, we obtain
\begin{align}
&2(C-D)zq'
+(C-D)^2(q-1)^2 \nonumber\\
&\quad +(C-D)(q-1)(\alpha+1-z)A
-(n+1)zA^2=0.
\end{align}

Now expand each term.

\medskip

\[
(q-1)^2=q^2-2q+1.
\]

\medskip

\begin{align*}
(q-1)A
&=(q-1)\big((1+D)q+(1-D)\big) \\
&=(1+D)q^2-2Dq-(1-D).
\end{align*}

\medskip

\begin{align*}
A^2
&=((1+D)q+(1-D))^2 \\
&=(1+D)^2q^2+2(1-D^2)q+(1-D)^2.
\end{align*}

\medskip

Substituting these expansions and collecting like terms, we obtain
\begin{align}\label{eq2}
&2(C-D)zq'(z) \nonumber\\
&+\Big[(C-D)^2+(C-D)(1+D)(\alpha+1-z)-(n+1)z(1+D)^2\Big]q^2(z) \nonumber\\
&+\Big[-2(C-D)^2-2D(C-D)(\alpha+1-z)-2(n+1)z(1-D^2)\Big]q(z) \nonumber\\
&+\Big[(C-D)^2-(C-D)(1-D)(\alpha+1-z)-(n+1)z(1-D)^2\Big]=0.
\end{align}

Hence, the equation takes the form
\[
F_1 zq'(z) + F_2 q(z)^2 + F_3 q(z) + F_4 = 0,
\]
where
\begin{align*}
F_1 &= 2(C-D), \\
F_2 &= (C-D)^2+(C-D)(1+D)(\alpha+1-z)-(n+1)z(1+D)^2, \\
F_3 &= -2(C-D)^2-2D(C-D)(\alpha+1-z)-2(n+1)z(1-D^2), \\
F_4 &= (C-D)^2-(C-D)(1-D)(\alpha+1-z)-(n+1)z(1-D)^2.
\end{align*}

With the help of \eqref{eq2}, $\Delta=0$ imply that $\psi(q(z),zq^{\prime}(z);z)\in \Delta$. Assume that
$G_{1}=\Re(F_{1})=2(C-D)$, $G_{2}=$\\$$\Re(F_{2})=(C-D)^2+(C-D)(1+D)(\alpha+1-x)-(n+1)x(1+D)^2,$$
$$G_{3}=\Re(iF_{3})=-2(C-D)^2-2D(C-D)(\alpha+1+y)+2(n+1)y(1-D^2),$$
$$G_{4}=(C-D)^2-(C-D)(1-D)(\alpha+1-x)-(n+1)x(1-D)^2.$$
For $\sigma\leq -\frac{(1+\rho^{2})}{2}$, $\rho\in \mathbb{R}$.
\begin{align*}
&\Re\psi(i\rho,\sigma; z)=G_{1}\sigma+G_{2}(i\rho)^{2}+G_{3}\rho+G_{4}\\
& = G_{1}\sigma-G_{2}(\rho)^{2}+G_{3}\rho+G_{4}\\
&\leq G_{1}\left(\frac{-(1+\rho^{2})}{2}\right)-G_{2}\rho^{2}+G_{3}\rho+G_{4}\\
&\leq \frac{-G_{1}}{2}-\frac{G_{1}\rho^{2}}{2}-G_{2}\rho^{2}+G_{3}\rho+G_{4}\\
&=\frac{-G_{1}-G_{1}\rho^{2}-2G_{2}\rho^{2}+2G_{3}\rho+2G{4}}{2}\\
&=\frac{-1}{2}[(G_{1}+2G_{2})\rho^{2}-2G_{3}\rho+G_{1}-2G_{4}]=R(\rho).
\end{align*}
By given hypothesis from any of  \eqref{eq4} and \eqref{eq5} $(G_{1}+2G_{2})>0$, $R(\rho)$ has maximum value at $\rho=\frac{G_{3}}{G_{1}+2G_{2}}$. Now finding value of $R(\rho)$ at $\rho=\frac{G_{3}}{G_{1}+2G_{2}}$ , for all $ \rho, |x|,|y|<1$, we have
\begin{align*}
&R(\rho)=\frac{-1}{2}[(G_{1}+2G_{2})\left(\frac{G_{3}}{G_{1}+2G_{2}}\right)^{2}-2G_{3}\left(\frac{G_{3}}{G_{1}+2G{2}}\right)-2G_{4}+G_{1}]<0\\
&=\frac{-1}{2}\left[\frac{-(G_{3})^{2}}{G_{1}+2G_{2}}-2G_{4}+G_{1}\right]< 0\\
&=\frac{G_{3}^{2}}{G_{1}+2G_{2}}< G_{1}-2G_{4}\\
&=G_{3}^{2}< (G_{1}+2G_{2})(G_{1}-2G_{4}).
\end{align*}
For $|x|<1,|y|<1$ and $y^{2}<1-x^{2}$ for above inequality, we have to prove
\begin{align}\label{eq3}
&[-(C-D)^2 - D(C-D)(\alpha+1+y) + (n+1)y(1-D^2)\Big]^2\nonumber\\
&<
\Big[(C-D)+(C-D)^2+(C-D)(1+D)(\alpha+1-x)-(n+1)x(1+D)^2\Big]\nonumber\\
&\times \Big[(C-D)-(C-D)^2+(C-D)(1-D)(\alpha+1-x)+(n+1)x(1-D)^2\Big].
\end{align}
Now, Verification of the required inequality,

\begin{align*}
L(y)
&= -(C-D)^2 - D(C-D)(\alpha+1+y) + (n+1)y(1-D^2).
\end{align*}

Let $K=C-D$. Then
\begin{align*}
L(y)
&= -K^2 - DK(\alpha+1+y) + (n+1)y(1-D^2)\\
&= -\big[K^2 + DK(\alpha+1)\big]
+ y\big[(n+1)(1-D^2)-DK\big].
\end{align*}

Define
\[
A=K^2 + DK(\alpha+1), \quad
B=(n+1)(1-D^2)-DK.
\]

Thus,
\[
L(y) = -A + By.
\]

\medskip

 Determine the maximum value of left hand side .
Since $y^2<1-x^2 \le 1$, we have $y \in (-1,1)$.
Hence,
\[
\max_{|y|<1} L(y)^2 = \max\{(-A+B)^2,(A+B)^2\}.
\]

That is,
\begin{align*}
\max L(y)^2
&= \max \Big\{
\big[-K^2 - DK(\alpha+2) + (n+1)(1-D^2)\big]^2,\\
&\qquad\qquad
\big[K^2 + DK\alpha + (n+1)(1-D^2)\big]^2
\Big\}.
\end{align*}

\medskip

Now, we determine the minimum value of right hand side

Let
\begin{align*}
H(x)
&=\Big[(C-D)+(C-D)^2+(C-D)(1+D)(\alpha+1-x)-(n+1)x(1+D)^2\Big]\\
&\quad \times \Big[(C-D)-(C-D)^2+(C-D)(1-D)(\alpha+1-x)+(n+1)x(1-D)^2\Big].
\end{align*}

Writing $K=C-D$, we express
\[
H(x)=h_1 x^2 + h_2 x + h_3,
\]
where
\begin{align*}
h_1
&=\Big[-(C-D)(1+D)-(n+1)(1+D)^2\Big]
\Big[-(C-D)(1-D)+(n+1)(1-D)^2\Big],
\end{align*}

\begin{align*}
h_2
&=\Big[(C-D)+(C-D)^2+(C-D)(1+D)(\alpha+1)\Big]
\Big[-(C-D)(1-D)+(n+1)(1-D)^2\Big] \\
&\quad + \Big[-(C-D)(1+D)-(n+1)(1+D)^2\Big]
\Big[(C-D)-(C-D)^2+(C-D)(1-D)(\alpha+1)\Big],
\end{align*}

\begin{align*}
h_3
&=\Big[(C-D)+(C-D)^2+(C-D)(1+D)(\alpha+1)\Big]
\Big[(C-D)-(C-D)^2+(C-D)(1-D)(\alpha+1)\Big].
\end{align*}

By given condition $h_1>0$, then the minimum occurs at
\[
\nu =-\frac{h_2}{2h_1},
\] since $H'(x)=0$ at $-\frac{h_2}{2h_1}$ and $H^{\prime\prime}(x)=2>0.$ 
Therefore minimum of H(x) will be $h_3 - \dfrac{h_2^2}{4h_1}$ for $\nu  \in (-1,1)$. In this case $|\nu|\geq 1$ and $H^{\prime}(x)=2h_{1}x+h_{2}\leq 2h_{1}+h_{2}\leq 0$ equivalent to given condition, $H(x)$
is monotonic decreasing. Hence $H(x)\geq H(1)=h_{1}+h_{2}+h_{3}\geq 0$. Thus, 
\[
\min_{|x|<1} H(x) =
\begin{cases}
h_3 - \dfrac{h_2^2}{4h_1}, & \nu  \in (-1,1),\\[6pt]
h_1+h_2+h_3, & \text{otherwise}.
\end{cases}
\]

We obtain equivalent to given conditions of the theorem  ,
\begin{align*}
&\max \Big\{
\big[-(C-D)^2 - (C-D)K(\alpha+2) + (n+1)(1-D^2)\big]^2,\\
&\qquad
\big[(C-D)^2 + D(C-D)\alpha + (n+1)(1-D^2)\big]^2
\Big\}\\
&\le \min_{|x|<1} H(x).
\end{align*}
Which proves the required result. 

\end{proof}
\begin{corollary}
    Let us consider the relation, we have
    $$zM_{n,\alpha}(z)=-\frac{\alpha}{n+1}zM'_{n+1,\alpha-1}(z),$$
    and consequently $$z\frac{(zM_{n,\alpha}(z))'}{zM_{n,\alpha}(z)}=1+z\frac{M^{\prime\prime}_{n+1,\alpha-1}(z)}{M^{\prime}_{n+1,\alpha-1}(z)}$$

Simultaneously with the Theorem \ref{thm2} and substituting $n=n-1,\alpha=\alpha+1$, yields the following reuslt as $zM_{n,\alpha}(z)\in S^{*}[C,D]$ (Janowski starlike). 
\end{corollary}

\section{Conclusion}
This work establishes new geometric properties of the Laguerre polynomial through the framework of exponential subordination and Janowski function theory. Sufficient conditions have been obtained for Janowski starlikeness and Janowski convexity, extending and complementing earlier results on convexity, close-to-convexity, strong starlikeness, and strong convexity. The derived results provide a broader understanding of the geometric behavior of Laguerre polynomials and strengthen their connection with geometric function theory. The included examples and corollaries illustrate the applicability and validity of the main results. It is expected that these findings may motivate further investigations on other special functions and related subclasses defined via differential subordination techniques.

{}

\end{document}